\documentclass[10pt]{article}
\font\smallit=cmti10  
\usepackage{amssymb,latexsym,amsmath,epsfig,amsthm} 

\makeatletter

\renewcommand\section{\@startsection {section}{1}{\z@}
{-30pt \@plus -1ex \@minus -.2ex} {2.3ex \@plus.2ex}
{\normalfont\normalsize\bfseries}}

\renewcommand\subsection{\@startsection{subsection}{2}{\z@}
{-3.25ex\@plus -1ex \@minus -.2ex} {1.5ex \@plus .2ex}
{\normalfont\normalsize\bfseries}}

\renewcommand{\@seccntformat}[1]{\csname the#1\endcsname. }
\newtheorem{thm}{Theorem}[section]
\newtheorem{lem}[thm]{Lemma}
\newtheorem{cor}[thm]{Corollary}
\theoremstyle{definition}
\newtheorem{defn}[thm]{Definition}
\theoremstyle{remark}
\newtheorem{remark}[thm]{Remark}
\theoremstyle{example}

\numberwithin{equation}{section}

\newcommand{\N}{{\mathbb N}}
\newcommand{\R}{{\mathbb R}}
\newcommand{\Q}{{\mathbb Q}}
\newcommand{\Z}{{\mathbb Z}}

 2

\makeatother

\begin{document}

\begin{center}
\uppercase{\bf New Classes of Infinite Image Partition Regular Matrices Near Zero} \vskip 20pt
{\bf Ram Chandra Manna\footnote{The author was supported by Swami Vivekananda Merit Cum Means Scholarship.}}\\
{\smallit Department of Mathematics, Ramakrishna Mission
Vidyamandira, Belur Math, Howrah, West bengal, India}\\
{\tt mannaramchandra8@gmail.com}
\\ {\bf Sourav Kanti Patra\footnote{The author was supported by UGC fellowship.}}\\
{\smallit Department of Mathematics, Ramakrishna Mission
Vidyamandira, Belur Math, Howrah, West bengal, India}\\
{\tt souravkantipatra@gmail.com}
\\ {\bf Rajib Sarkar\footnote{The author was supported by UGC fellowship.}}\\
{\smallit Department of Mathematics, Indian Institute of Technology Madras, Chennai, India}\\
{\tt rajib.sarkar63@gmail.com}
\vskip 10pt
\end{center}
\vskip 30pt

\centerline{\bf Abstract}

\noindent Image partition regular matrices near zero generalizes many classical results of Ramsey Theory. There are several characterizations of finite image partition regular matrices near zero. Contrast to the finite cases there are only few classes of matrices that are known to be infinite image partition regular near zero. In this present work we have produced several new examples of such classes.

\pagestyle{myheadings}
\thispagestyle{empty}
\baselineskip=12.875pt \vskip 30pt

\section{Introduction}

In this paper we shall be concerned with finite or infinite matrices with rational entries.
\begin{defn}\label{1.1}
A matrix (finite or infinite) is admissible if and only if it has entries from $\mathbb{Q}$, no row equal to $\vec{0}$  and finitely many nonzero entries in each row.
\end{defn}
Here the rows and columns of a matrix $A$ are indexed by ordinals (always countable). The first infinite ordinal is $\omega =\mathbb{N}\cup \{0 \}$. Recall that an ordinal is the set of its predecessors, so for an ordinal $\alpha$, the statements $x < \alpha$ and $x\in \alpha$ are synonymous. If $A$ is an $\alpha \times \delta$ matrix, $B$ is a $\gamma \times \tau$ matrix, then $\left(\begin{matrix} A & 0 \\ 0 & B \end{matrix}\right)$ is an $(\alpha +\gamma)\times (\delta +\tau)$ matrix where $0$ represents a matrix with all entries equal to 0 of the appropriate size.  $\left(\begin{matrix} A & 0 \\ 0 & B \end{matrix}\right)$ is called the diagonal sum of $A$ and $B$.
\begin{defn}\label{1.2} (see \cite[Definition 1.2]{RefDN08})
Let $S$ be a subsemigroup of ($\mathbb{R}$,+) and let $\alpha$, $\delta$ be positive ordinals. An $\alpha \times \delta$ matrix $A$ is image partition regular over $S$ if and only if $A$ is admissible and whenever $S\setminus \{0 \}$ is finitely colored, there is some $\vec{x}\in S^{\delta}$ such that the entries of $A\vec{x}$ are monochromatic.
\end{defn}
Image partition regular matrices generalize many classical results of Ramsey Theory. For example, Schur's Theorem \cite{RefS} and the Van der Waerden's Theorem \cite{RefW} are equivalent to say that the matrices $\left(\begin{matrix} 1 & 0 \\ 0 & 1\\ 1 & 1
\end{matrix}\right)$ and for each $n\in \mathbb N$,
$\left(\begin{matrix} 1 & 0 \\ 1 & 1
\\\vdots & \vdots \\ 1 & n-1 \end{matrix}\right)$ are image partition regular over $\mathbb{N}$. We call any finite or infinite matrix $A$ to be image partition regular if and only if it is image partition regular over $\mathbb{N}.$ There are several characterizations of finite image partition regular matrices. One of these characterizations involves the notion of central sets. Central sets were introduced by Furstenberg and defined in terms of notion of topological dynamics. A nice characterization of central sets in terms of algebraic structure of $\beta\mathbb{N}$, the Stone-$\breve{C}$ech compactification of $\mathbb{N}$ is given in Definition \ref{1.4} . Central sets are very rich in combinatorial properties. The basic fact about central sets from Central Sets Theorem that we need to know is given by Furstenberg \cite[Proposition 8.21]{RefF}.
\begin{thm}\label{1.3}
(Central Sets Theorem) Let (S,+) be a commutative semigroup. Let
$\tau$ be the set of sequences $\langle y_t\rangle _{t=1}^\infty$ in $S$. Let $C$ be a subset of $S$ which is central and let $F\in\mathcal P_{f}(\tau)$. Then there exists a sequence 
$\langle a_t\rangle _{t=1}^\infty$
 in $S$ and a sequence
  $\langle H_t\rangle _{t=1}^\infty$ in 
  $\mathcal P_{f}(\mathbb N)$ such that for each 
  $n\in\mathbb N$, $\max H_n < \min H_{n+1}$ and for each $L\in\mathcal P_{f}(\mathbb N)$ and each 
  $f\in F$, $ \displaystyle \sum_{n\in
L}\Big(a_n +\displaystyle \sum_{t\in H_n} f(t) \Big)\in C$.
\end{thm}
We shall present this algebraic characterization of central sets below, after introducing the necessary background information. \\
     Given a discrete space $S$, we take the Stone-$\breve{C}$ech compactification $\beta S$ of $S$ to be the set of ultrafilters on $S$, the principal ultrafilters being identified with the points of $S$. Given $A\subseteq S$, let $\bar{A}=\{p\in\beta S:A\in p\}$. Then $\{\bar{A}:A\subseteq S\}$ forms a basis for open sets as well as closed sets of $\beta S$. If $(S,\cdot)$ is discrete space, the operation extends to $\beta S$ so that $(\beta S,\cdot)$ becomes a compact right topological 
semigroup with $S$ contained in its topological center. That is, for any $p\in \beta S$, the function $q\mapsto q\cdot p$ from $\beta S$ to itself is continuous and for any
$x \in S$, the function $q\mapsto x\cdot q$ from $\beta S$ to itself is continuous. Given $p,q\in\beta S$ and $A\subseteq S, A\in p\cdot q$ if and only if $\{x\in
S:x^{-1}A\in q\}\in p$, where $x^{-1}A=\{y\in S:x\cdot y\in A\}$.\\
A nonempty subset $I$ of a semigroup $(T,\cdot)$ is called a left ideal of $T$ if 
$T\cdot I\subseteq I$, a right ideal if $I\cdot
T\subseteq I$, and a two-sided ideal (or simply an ideal) if it is both a left and a right ideal. A minimal left ideal is a left
ideal that does not contain any proper left ideal. Similarly, we
can define minimal right ideal. Any compact Hausdorff right topological semigroup ($T, \cdot)$ has a smallest ideal $K(T)$ which is the union of all minimal left ideals of $T$ and is also union of all minimal right ideals of $T$. An idempotent belonging to $K(T)$ is called a minimal idempotent. Given $J\subseteq T$, we shall use the notation $E(J)$ to denote the set of all idempotents in $J$.
\begin{defn}\label{1.4} \cite[Definition 4.42]{RefHS18}
Let $S$ be a discrete semigroup and let $A\subseteq S$. Then $A$ is central in $S$ if and only if $A$ is a member of some idempotent in $\beta S$. 
\end{defn}
We now define the notion of image partition regularity near zero \cite[Definition 1.3]{RefDN08} and \cite[Definition 3.1]{RefDN08} .
\begin{defn}\label{1.5}
Let $S$ be a subsemigroup of ($\mathbb{R}$,+) with $0\in cl \ S$. Let $u,v \in \mathbb{N}$ and let $A$ be a $u \times v$ matrix with entries from $\mathbb{Q}$. Then $A$ is image partition regular over $S$ near zero if and only if, whenever $S \setminus \{0\}$ is finitely colored and $\delta >
 0$, there exists $\vec{x} \in S^{u}$ such that the entries of $A\vec{x}$ are monochromatic and lie in the interval $(-\delta, \delta).$ 
\end{defn}
In \cite{RefPG17}, the authors produced a class of infinite image partition regular matrices which are compatible with Milliken-Taylor matrices ( a class of infinite image partition regular matrices, defined in Definition \ref{3.7}(a)) with respect to diagonal sum. Recently in \cite{RefHS18} , several new examples of infinite partition regular matrices were obtained by Hindman and Strauss. In Section 3, using these concepts we produce several new examples of infinite image partition regular matrices near zero. In Section 2, we develope some tools for Section 3.
\section{Ultrafilters near zero induced by matrices}
Hindman and Leader first introduced the notion of ultrafilters near zero and studied many classical Ramsey Theoritic properties near zero in \cite{RefHL99}.
 We now recall \cite[Definition 4.2]{RefDN08}.
\begin{defn}\label{2.1}
Let $S$ be a dense subsemigroup of ($\mathbb{R}$,+) or of $((0,\infty),+).$ Then $O^{+}(S)=\{p\in \beta S: (\forall  \epsilon > 0) ((0, \epsilon) \cap S \in p)\}.$ If $S$ is a dense subsemigroup of ($\mathbb{R}$,+), then $O^{-}(S)=\{p\in \beta S: (\forall \epsilon > 0) ((-\epsilon,0) \cap S \in p)\}.$
\end{defn}
It was shown in \cite[Lemma 2.5]{RefHL99} that $O^{+}(S)\cap K(\beta S)=\emptyset$, so we can not obtain any information about $K(O^{+}(S))$ based on knowledge of $K(\beta S)$. Thus one can define different notions of large sets near zero using $O^{+}(S).$
\begin{defn}\label{2.2} \cite[Definition 4.1(a)]{RefHL99}
Let $S$ be a dense subsemigroup of $((0,\infty),+).$ A set $C$ is central near zero if and only if there is an idempotent $p\in \bar{C}\cap K(O^{+}(S)).$ 
\end{defn}
We now define the notion of thick sets near zero.
\begin{defn}\label{2.3}
Let $S$ be a dense subsemigroup of $((0,\infty),+).$ A set $C$ is thick near zero if and only if there is a left ideal $L$ of $O^{+}(S)$ such that $L\subseteq \bar{C}.$
\end{defn}
We shall now concentrate on an interesting subset of $O^{+}(S)$, induced by admissible matrices.
\begin{defn}\label{2.4}
Let $S$ be a subsemigroup of $((0,\infty),+)$ and let $A$ be a finite or infinite matrix with entries from $\mathbb{Q}.$ Then
$ I(A;S)=\{p\in \beta S: \text{for every} \ P \in p, \ \text{there exists} \ \vec{x} \ \text{with entries from} \ S \ \text{such that all entries of} \ A\vec{x} \ \text{are} \ \text{in} \  P \}.  $
\end{defn}
In this article we use the notation $I(A)$ to denote $I(A;\N)$.
\begin{lem}\label{2.5}
Let $S$ be a subsemigroup of $((0,\infty),+)$. Let $u,v\in \mathbb{N}\cup \{\omega \}$ and let $A$ be a $u\times v$ matrix with entries from $\mathbb{Q}$.
\\ (a) The set $I(A;S)$ is compact and $I(A;S)\neq \emptyset$ if and only if $A$ is image partition regular over $S$.
\\ (b) If $A$ is finite image partition regular over $S$, then $I(A)$ is a subsemigroup of $(\beta S,+)$.
\end{lem}
\begin{proof}
See the proof of \cite[Lemma 2.5]{RefHLS03} .
\end{proof}
\begin{defn}\label{2.6}
Let $S$ be a dense subsemigroup of $((0,\infty),+)$ and let $A$ be a finite or infinite matrix with entries from $\mathbb{Q}$. Then
$ I_{0}(A;S)=\{p\in \beta S: \text{for every} \ P \in p \ \text{every} \ \epsilon>0 \ \text{there exists} \ \vec{x} \ \text{with entries from} \ S \ \text{such that all entries of} \ A\vec{x} \ \text{are} \ \text{in} \ P\cap (0,\epsilon) \}. $
\end{defn}
\begin{remark}\label{2.7}
Let $S$ be a dense subsemigroup of ($\mathbb{R},+)$ and let $A$ be a finite or infinite matrix with entries from $\mathbb{Q}$.  Then $I_{0}(A;S)=I(A;S)\cap O^{+}(S).$
\end{remark}
\begin{lem}\label{2.8}
Let $S$ be a dense subsemigroup of $((0,\infty),+)$. Let $u,v \in \mathbb{N}\cup \{\omega \}$ and let $A$ be a $u\times v$ matrix with entries from $\mathbb{Q}$.
\\ (a) The set $I_{0}(A;S)$ is compact and $I_{0}(A;S)\neq \emptyset$ if and only if $A$ is image partition regular near zero over $S$.
\\ (b) If $A$ is finite image partition regular near zero over $S$, then $I_{0}(A;S)$ is a subsemigroup of $\beta S$.
\end{lem}
\begin{proof}
Proof of this lemma immediately follows from Remark \ref{2.7} and Lemma \ref{2.5}.
\end{proof}
For any subsemigroup $S$ of $((0,\infty),+)$, there is a natural action of $\mathbb{N}$ on $\beta S$.
\begin{defn}\label{2.9} (see \cite[Definition 1.1]{RefHS00(1)})
Let $(S,+)$ be a subsemigroup of ($\mathbb{R},+)$ and $n\in \mathbb{N}$. Define $l_{n}(s)=n\cdot s$ and let $\tilde{l_{n}}: \beta S\longrightarrow \beta S$ be the continuous extension of $l_{n}$. For $p\in \beta S$, define $n\cdot p=\tilde{l_{n}}(p).$
\end{defn}
\begin{lem}\label{2.10}
Let $(S,+)$ be a subsemigroup of $((0,\infty),+)$, let $a\in \mathbb{N}$, let $p\in \beta S$ and let $A\subseteq S$. Then $A\in a\cdot p$ if and only if $a^{-1}A\in p$ where $a^{-1}A=\{ x\in S: a\cdot x\in A\}.$ In particular, if $B\in p$, then $a\cdot B \in a\cdot p.$
\end{lem}
\begin{proof}
For the proof see \cite[Lemma 2.1]{RefHS00(1)}.
\end{proof}
There is also a natural action of $\beta \mathbb{N}$ on $\beta S$, denoted by $q\cdot p$ for $p\in \beta S$ and $q\in \beta \mathbb{N}$, which is given by $q\cdot p=q-lim_{n\in \mathbb{N}} (n\cdot p).$
Note that for a dense subsemigroup $S$ of $((0,\infty),+)$,  for all $n\in \mathbb{N}$, $q\in \beta \mathbb{N}$ and $p\in O^{+}(S)$, we have $n\cdot p\in O^{+}(S)$ and $q\cdot p\in O^{+}(S).$ As a consequence of previous lemma, we have the following:
\begin{lem}\label{2.11}
Let $S$ be a dense subsemigroup of $((0,\infty),+)$ and $A$ be a finite or infinite matrix with entries from $\mathbb{Q}.$
\\ (a) If $q\in \beta \mathbb{N}$ and $p\in I_{0}(A;S)$ then $q\cdot p\in I_{0}(A;S)$.
\\ (b) If $A$ is finite, $q\in I(A)$ and $p\in O^{+}(S)$, then $q\cdot p\in I_{0}(A;S)$.
\end{lem}
\begin{proof}
Let $u,v\in \mathbb{N}\cup \{ \omega \}$ and $A$ be a $u\times v$ matrix. Note that for all $q\in \beta \mathbb{N}$  and $p\in O^{+}(S)$, $q\cdot p \in O^{+}(S)$, so it suffices to show that $q\cdot p\in I(A;S)$ for both of the cases.
\\ (a) Let $q\in \beta \mathbb{N}$ and $p\in I_{0}(A;S)$. Suppose $E\in q\cdot p$. Then $\{n\in \mathbb{N}:n^{-1}E\in p\}\in q$. Choose $a\in \mathbb{N}$ such that $a^{-1}E\in p$. Since $p\in I(A;S)$, take $\vec{x}\in S^{v}$ such that $A\vec{x}\in (a^{-1}E)^{v}$. Then $A(a\vec{x})\in E^{v}$, as required.
\\ (b) Let $u,v\in \mathbb{N}$ and $A$ be a $u\times v$ matrix, let $q\in I(A)$ and let $p\in O^{+}(S)$. Suppose $E\in q\cdot p$ and let $W=\{n\mathbb{N} :n^{-1}E\in p \}$. Then $W\in q$. Now choose $\vec{x}\in S^{v}$ such that $\vec{y}=A\vec{x}\in W^{u}$. Then $\bigcap_{i=1}^{u} y_{i}^{-1}E \in p$. Take $a\in \bigcap_{i=1}^{u} y_{i}^{-1}E$ and $\vec{z}=a\vec{x}$. Then $A\vec{z}\in E^{v}$, as required.
\end{proof}
We now give an alternative proof of \cite[Lemma 2.1]{RefDN08} in this context.
\begin{thm}\label{2.12}
Let $u,v\in \mathbb{N}$, let $A$ be a $u\times v$ matrix with entries from $\mathbb{Q}$ such that $A$ is image partition regular and let $S$ be a dense subsemigroup of $((0,\infty),+)$. Then $A$ is image partition regular near zero over $S$.
\end{thm}
\begin{proof}
By Lemma \ref{2.5}(a), $I(A)\neq \emptyset$. By Lemma \ref{2.11}(b), choose $r\in I_{0}(A;S)$ and therefore $I_{0}(A;S)\neq \emptyset$. By Lemma \ref{2.8}(a), one can conclude that $A$ is image partition regular near zero over $S$.
\end{proof}
In the following theorem, we see that the notion of image partition regularity for different subsemigroups of $((0,\infty),+)$ and image partition regularity near zero are same for finite matrices.
\begin{thm}\label{2.13}
Let $u,v \in \mathbb{N}$, and let $A$ be a $u\times v$ matrix with entries from $\mathbb{Q}$ and let $S$ be a dense subsemigroup of $((0,\infty),+)$. The following statements are equivalent:
\\ (a) $A$ is image partition regular.
\\ (b) $A$ is image partition regular near zero over $S$.
\\ (c) $A$ is image partition regular over $S$.
\\ (d) $A$ is image partition regular over $((0,\infty),+)$.
\end{thm}
\begin{proof}
The proof follows immediately from \cite[Corollary 1.6] {RefDP08}.
\end{proof}
\begin{thm}\label{2.14}
Let $S$ be a dense subsemigroup of $((0,\infty),+)$ for which $cS$ is $central^{*}$ near zero for each $c\in \mathbb{N}$. Let $u,v \in \mathbb{N}$ and let $A$ be a $u\times v$ matrix with entries from $\mathbb{Q}$. The following statements are equivalent:
\\ (a) $A$ is image partition regular.
\\ (b) For each subset $C$ of $S$, which is central near zero, there exists $\vec{x}\in S^{v}$ such that $A\vec{x}\in C^{u}.$
\end{thm}
\begin{proof}
(a)$\Rightarrow$ (b) Suppose $A$ is image partition regular. 
By \cite[Theorem 2.10(f)]{RefHLS02}, there exist $m\in \mathbb{N}$, a $v\times m$ matrix $G$ with entries from $\omega$ and no row equal to $\vec{0}$, a $v\times m$ first entries matrix $B$ with entries from $\omega$, and $c\in \mathbb{N}$ such that $c$ is the only first entry of $B$ and $AG=B$. By \cite[Theorem 5.3]{RefHL99}, choose $\vec{y}\in S^{m}$ such that $B\vec{y}\in C^{u}$. Let $\vec{x}=G\vec{y}$. Then $A\vec{x}=AG\vec{y}=B\vec{y} \in C^{u}$, as required.
\\ (b)$\Rightarrow$ (a) Trivially follows from Theorem \ref{2.12}.
\end{proof}
\begin{thm}\label{2.15}
Let $A$ and $B$ be finite and infinite image partition regular matrices respectively ( with rational coefficients). Then $\left(\begin{matrix} A & 0 \\ 0 & B \end{matrix}\right)$ is image partition regular.
\end{thm}
\begin{proof}
See the proof of \cite[Lemma 2.3]{RefHLS03}  or \cite[Theorem 2.4]{RefPG17}.
\end{proof}
Following theorem is an analogue of the above theorem.
\begin{thm}\label{2.16}
Let $S$ be a dense subsemigroup of $((0,\infty),+)$. Let $A$ and $B$ be finite and infinite image partition regular matrices over $S$ respectively. Then $\left(\begin{matrix} A & 0 \\ 0 & B \end{matrix}\right)$ is image partition regular near zero over $S$.
\end{thm}
\begin{proof}
Let $u,v \in \mathbb{N}$ and $A$ be a $u\times v$ matrix and $M=\left(\begin{matrix} A & 0 \\ 0 & B \end{matrix}\right)$. By Theorem   \ref{2.13}, $A$ is image partition regular. Now take $q\in I(A)$ and $p\in I_{0}(B;S)$. Then by Lemma \ref{2.11}(b), $q\cdot p \in I_{0}(A;S)$ and by Lemma \ref{2.11}.(a), $q\cdot p \in I_{0}(B;S)$. Let $k\in \mathbb{N}$ be given and $S=\bigcup_{i=1}^{k} E_{i}$ and let $\delta >0$. Choose $i\in \{1,2,\ldots,k\}$ such that $E_{i}\in q\cdot p$. So by definition of $I_{0}(A;S)$ and $I_{0}(B;S)$, there exists $\vec{x}\in S^{v}$ such that $A\vec{x} \in (E_{i}\cap(0,\delta))^{u}$ and $\vec{y}\in S^{\omega}$ such that $B\vec{y}\in (E_{i}\cap(0,\delta))^{\omega}.$ Let $\vec{z}=\left(\begin{matrix} \vec{x}
\\ \vec{y} \end{matrix}\right),$ then $M\vec{z}=\left(\begin{matrix} A\vec{x}
\\ B\vec{y} \end{matrix}\right)$. So $M\vec{z}\in (E_{i}\cap(0,\delta))^{\omega}$. Therefore $M=\left(\begin{matrix} A & 0 \\ 0 & B \end{matrix}\right)$ is image partition regular near zero over $S$.
\end{proof}
\begin{thm}\label{2.17}
Let $S$ be a dense subsemigroup of $((0,\infty),+)$ and let $\mathcal{F}$ be a class of finite image partition regular matrix with entries from $\mathbb{Q}$. If $B$ is an image partition regular matrix near zero over $S$. Then $I(B;S)\cap \big( \bigcap_{A\in \mathcal{F}}I(A;S) \big)\neq \emptyset$.
\end{thm}
\begin{proof}
Let $\mathcal{G}$ be a finite subclass of $\mathcal{F}$. Now by Lemma \ref{2.11}(b), we can conclude that $I(B;S)\cap \big( \bigcap_{A\in \mathcal{G}}I(A;S) \big)\neq \emptyset$ and the result follows from the compactness argument.
\end{proof}
\section{Subtracted image partition regularity near zero}
In \cite{RefHS18} authors produced several new examples of infinite image partition regular matrices. Following them we have produced some new examples of infinite image partition regular matrices near zero. We now recall the following definition \cite[Definition 1.5]{biswaspaul}.
\begin{defn}\label{3.1}
Let $S$ be  a dense subsemigroup of $((0,\infty),+)$ and let $A$ be an $\omega\times \omega$ matrix with entries from $\mathbb{Q}$. Then $A$ is centrally image partition regular near zero if and only if whenever $C$ is a central set near zero in $S$, there exists $\vec{x}\in S^{\omega}$ such that $A\vec{x}\in C^{\omega}.$
\end{defn}
We now introduce the following definition of subtracted image partition regular matrices near zero over $S$.
\begin{defn}\label{3.2}
Let $S$ be a dense subsemigroup of $((0,\infty),+)$ and let $A$ be an $\omega\times \omega$ matrix with entries from $\mathbb{Q}$. The matrix $A$ is subtracted image partition regular near zero over $S$, if and only if 
\\ (1) no row of $A$ is $\vec{0}$.
\\ (2) for each $i\in \omega$, $\{j\in \omega:a_{ij}\neq 0\} $ is finite, and
\\ (3) there exists $v\in \mathbb{N}$, an $\omega \times v$ matrix $A_{1}$ and an $\omega \times \omega$ matrix $A_{2}$ such that rows of $A_{1}$ are the rows of a finite image partition regular matrix, $A_{2}$ is a image partition regular matrix near zero over $S$ and $A=(A_{1} \ A_{2})$.
\end{defn}
In the above definition if $A_{2}$ is centrally image partition regular near zero over $S$, then $A$ is said to be subtracted centrally image partition regular matrix near zero over $S$ and if $A_{2}$ is Milliken-Taylor matrix near zero (Definition \ref{3.7}(b)) then $A$ is said to be subtracted Milliken-Taylor matrix near zero.
\begin{thm}\label{3.3}
Let $S$ be a dense subsemigroup of $((0,\infty),+)$ and let $A$ be a subtracted image partition regular matrix near zero. Then $A$ is image partition regular near zero.
\end{thm}
\begin{proof}
Since $A$ is a subtracted image partition regular matrix near zero over $S$, pick 
$u,v\in\mathbb N$, a $u\times v$ image partition regular matrix $D$, an $\omega\times v$ 
matrix $A_1$ whose rows are all rows of $D$, and an $\omega \times \omega$ image partition 
regular matrix $A_2$ such that $A=\left(\begin{matrix} A_1 & A_2 \end{matrix}\right)$ as 
in the Definition \ref{3.2}. Note that $I(A_1)\cap I(A_2)=I(D)\cap I(A_2)\neq\emptyset$ by 
Theorem \ref{2.17}. Choose $p\in I(A_1)\cap I(A_2)$. Let $r\in\mathbb N$ and 
$\mathbb N=\bigcap_{i=1}^rE_i$ be a $r$-coloring of $\mathbb N$. Pick $k\in\{1, 2, ....., r\}$ 
such that $E_k\in p+p$. Then $U=\{x\in\mathbb N:-x+E_k\in p\}\in p$. Pick $\vec{x}^{(1)}\in\mathbb{N}^v$ 
such that $A_1\vec{x}^{(1)}\in U$. If $\vec{y}=A\vec{x}^{(1)}$ then $\{y_j:j<\omega\}$ is finite 
so $V=\bigcap_{i<\omega}(-y_i+E_k)\in p$. Choose $\vec{x}^{(2)}\in\mathbb{N}^\omega$ 
such that $\vec w=A_2\vec{x^{(2)}} \in V^\omega$. Let 
$\vec{x}=\left(\begin{matrix} \vec{x}^{(1)} \\ \vec{x}^{(2)} \end{matrix}\right)$. 
Then $\vec{x}\in \mathbb{N}^\omega$ and 
$A\vec{x}=A_1\vec{x}^{(1)}+A_2\vec{x}^{(2)}=\vec{y}+\vec{w}\in E_k^\omega$. 
Therefore $A$ is image partition regular near zero over $S$.
\end{proof}
\begin{thm}\label{3.4}
Let $S$ be a dense subsemigroup of $((0,\infty),+)$ for which $cS$ is $central^{*}$ near zero for each $c\in \mathbb{N}$. Let $A$ be a subtracted centrally image partition regular matrix near zero over $S$. Then $A$ is centrally image partition regular near zero over $S$.
\end{thm}
\begin{proof}
Since $A$ is a subtracted centrally image partition regular
matrix near zero over $S$, pick $u,v\in\mathbb N$, a $u\times v$ image partition
regular matrix $D$, an $\omega\times v$ matrix $A_1$ whose rows
are all rows of $D$, and an $\omega \times \omega$ centrally image
partition regular matrix $A_2$ such that $A=\left(\begin{matrix}
A_1 & A_2 \end{matrix}\right)$ as in Definition \ref{3.2}. Let $C$ be a
central set near zero in $S$ and $p$ be a minimal idempotent
of $O^{+}(S)$ such that $C\in p$. Let $B=\{x\in\mathbb
N:-x+C\in p\}$. Then $B\in p$ and hence $B$ is central in
$(\mathbb N,+)$. Now by Theorem \ref{2.14}, pick $\vec{x}^{(1)}\in
\mathbb{N}^v$ such that $D\vec{x}^{(1)}\in B^u$. If
$y=A_1\vec{x}^{(1)}$  then $\{y_i: i\in \omega\}$ is  finite so
$E=\bigcap_{i<\omega}(-y_i+C)\in p$. Therefore $E$ is central near zero in
$S$. Choose $\vec{x}^{(2)}\in \mathbb{N}^\omega$ such
that $\vec w=A_2\vec{x}^{(2)} \in E^\omega$. Let
$\vec{x}=\left(\begin{matrix} \vec{x}^{(1)} \\ \vec{x}^{(2)}
\end{matrix}\right)$. Then $\vec{x}\in \mathbb{N}^\omega$ and
$A\vec{x}=A_1\vec{x}^{(1)}+A_2\vec{x}^{(2)}=\vec{y}+\vec{w}\in
C^\omega$. Therefore $A$ is centrally image partition regular near zero over $S$.
\end{proof}
We now recall the following definitions \cite[Definition 2.1]{RefHS00}, \cite[Definition 2.3]{RefHS18}.
\begin{defn}\label{3.5}
Let $v\in \mathbb{N}\cup \{\omega \}$ and let $\vec{x}\in \mathbb{Z}^{v}$. Then 
\\ (a) $d(\vec{x})$ is the sequence obtained by deleting occurence of $0$ from $\vec{x}$.
\\ (b) $c(\vec{x})$ is the sequence obtained by deleting every digit in $d(\vec{x})$ which is equal to its predecessor and
\\ (c) $\vec{x}$ is a compressed sequence if and only if $\vec{x}=c(\vec{x})$.
\end{defn}
\begin{defn}\label{3.6}
Let $k\in \mathbb{N}$ and let $\vec{a}=<a_{1},a_{2},\ldots,a_{k}>$ be a compressed sequence in $\mathbb{Z}\setminus \{0\}$ with $a_{k}>0$. Then $MT(\vec{a})$ is a matrix consisting of all rows $\vec{r}$ with finitely many nonzero entries such that $C(\vec{r})=\vec{a}$.
\end{defn}
\begin{defn}\label{3.7}
(a) An $\omega
\times \omega$ matrix $M$ is said to be Milliken-Taylor matrix if there exists a compressed sequence $<a_{1},a_{2},\ldots,a_{k}>$ with $a_{k}>0$ such that $M=MT(\vec{a})$.
\\ (b) An $\omega
\times \omega$ matrix $M$ is said to be Milliken-Taylor matrix near zero if there exists a compressed sequence $<a_{1},a_{2},\ldots,a_{k}>$ with $a_{1}>0$, $a_{k}>0$ such that $M=MT(\vec{a})$.  
\end{defn}
\begin{lem}\label{3.8}
Let $S$ be a dense subsemigroup of $((0,\infty),+)$. Let $u,v \in \mathbb{N}$ and let $A$ be a $u\times v$ image partition regular matrix with entries from $\mathbb{Q}$ and let $p\in I(A;S)$. Then for each $D\in p$ and each $\epsilon>0$, there exists, $Q\subseteq (S\cap (0,\epsilon))^{v}$ such that for $\vec{z}\in Q$, $A\vec{z}\in D^{u}$. 
\end{lem}
\begin{proof}
Since $A$ is image partition regular, there exists $m\in \N$, a $v\times m$ matrix $G$ with entries from $\omega$ and no row equal to $\vec{0}$, a $v\times m$ first entries matrix $B$ with entries from $\omega$, and $c\in \N$ such that $c$ is the only first entry of $B$ and $AG=B$, by \cite[Theorem 2.10(f)]{RefHLS02}.
\\ So it suffices to assume that $A$ is a first entries matrix with entries from $\omega$. Let $\vec{p}=(p,p,\ldots,p)\in (\beta S)^{u}$. Define $f:S^{v}\rightarrow S^{u}$ by $f(x)=A\vec{x}$ and note that $f$ is a homomorphism. Let $Y=\{r\in \beta (S^{v}):(S\cap (0,\epsilon))^{v}\in r \ \text{for all} \ \epsilon >0\}.$ Let $\tilde{f}:\beta(s^{v})\rightarrow (\beta S)^{u}$ be the continuous extension of $f$. We claim that $\emptyset \neq f^{-1}[\{\bar{p}\}]\subseteq Y$. To see that $ f^{-1}[\{\bar{p}\}] \neq \emptyset$, note that $\{f^{-1}[B^{u}]:B\in p\}$ has the finite intersection property and since $p\in I_{0}(A;S)$, therefore $\bigcap_{B\in p} cl_{Y} f^{-1}[B^{u}]\neq \emptyset$ and it is routine to check that $\bigcap_{B\in p} cl_{Y} f^{-1}[B^{u}] \subseteq \tilde{f}^{-1}[\{\bar{p}\}]$. Now let $q\in \tilde{f}^{-1}[\{\bar{p}\}]\setminus Y$.  
\\ Then there exists $\epsilon>0$ such that $E=(S\cap(0,\epsilon))^{v}\in q\Rightarrow S^{v}\setminus E \in q\Rightarrow f(S^{v}\setminus E)\in \bar{p} \Rightarrow \pi_{j} (f(S^{v}\setminus E))\in p$. Let $a=\text{min}\{a_{ij}:1\leq i\leq u, \ 1\leq j\leq v\}$ where $A=(a_{ij})_{u\times v}$. Then $\pi_{j} (f(S^{v}\setminus E)) \subseteq S\setminus (0,a\epsilon)$ which is a contradiction. Now let $q\in \tilde{f}^{-1}[\{\bar{p}\}]\subseteq Y$. Then for each $D\in p$ and $\epsilon>0$, there exists $Q\subseteq (S\cap(0.\epsilon))^{v}$ such that $A\vec{z}\in D^{u}$ for all $\vec{z}\in Q$.
\end{proof}
From now onwards we need to view $O^{+}(S)$ as a subset of $\beta \R$. Given $p\in \beta \R$, we let $-p=\{-B:B\in p\}$. Then for all $p,q\in \beta \R$, $(-p)+(-q)=-(p+q)$, which follows from the homomorphism $\gamma :\R\rightarrow \R$ such that $\gamma(x)=-x$. For $p\in \beta \R$ and $n\in \N$, $(-n)\cdot p=-(n\cdot p)$. Also note that $p+(-q)\in O^{+}(S)$ for all $p,q\in O^{+}(S)$. 
\begin{lem}\label{3.9}
Let $S$ be a dense subsemigroup of $((0,\infty),+)$. Let $u,v\in \N$ and let $A$ be a $u\times v$ matrix with entries from $\Q$, which is image partition regular. Then the followings are true.
\\ (1) If $p\in I(A)$ and $r\in I(A)$, then $p+r\in I(A)$. 
\\ (2) If $p\in I(A)$ and $r\in -I(A)$, then $p+r\in I(A)$.
\\ (3) If $p\in I(A)$ and $r\in I(A)$, then $(-p)+r\in -I(A)$.
\\ (4) If $p\in I(A)$ and $r\in -I(A)$, then $(-p)+r\in -I(A)$. 
\end{lem}
\begin{proof}
The proof is a trivial consequence of the above discussion.
\end{proof}
\begin{lem}\label{3.10}
Let $S$ be a dense subsemigroup of $((0,\infty),+)$ such that $cS$ is an $IP^{*}$ set near zero. Let $p$ be an idempotent in $(\beta S,+)$ and let $\alpha \in \Q$ with $\alpha>0$. Then $\alpha\cdot p$ is also an idempotent in $(\beta S,+)$. If $p$ is a minimal idempotent then $\alpha\cdot p$ is also a minimal idempotent. Consequently, if $C$ is central near zero in $S$, then so is $\alpha C\cap S$.
\end{lem}
\begin{proof}
Follow the proof of \cite[Lemma 5.19.2]{RefHS98}.
\end{proof}
\begin{lem}\label{3.11}
Let $S$ be a dense subsemigroup of $((0,\infty),+)$ and for each $c\in \N$, $cS$ is an $IP^{*}$ set near zero. Let $A$ be a finite image partition regular matrix. Let $k\in \N$ and let $<a_{1},a_{2},\ldots,a_{k}>$ be a sequence in $\Z\setminus \{0\}$ with $a_{1}>0$, $a_{k}>0$. Let $p$ be a minimal idempotent in $O^{+}(S)$. Then $a_{1}\cdot p+a_{2}\cdot p+\cdots+a_{k}\cdot p\in I_{0}(A;S)$ and $\frac{a_{1}}{a_{k}}\cdot p+\frac{a_{2}}{a_{k}}\cdot p+\cdots+\frac{a_{k-1}}{a_{k}}\cdot p+p\in I_{0}(A;S)$.
\end{lem}
\begin{proof}
By Lemma \ref{2.11}(a), for each $i\in \{1,2,\ldots,k\}$, $|\frac{a_{i}}{a_{1}}|\cdot p$ and $|a_{i}|\cdot p$ are minimal idempotents in $O^{+}(S)$. Now by Lemma \ref{3.9}, for $i\in \{1,2,\ldots,k\}$, $a_{i}\cdot p$ is in $I_{0}(A;S)$ or $-I_{0}(A;S)$ according as $a_{i}$ is positive or negative. Conclusion of the lemma follows by repeated application of Lemma \ref{3.9}.
\end{proof}
\begin{lem}\label{3.12}
Let $S$ be a dense subsemigroup of $((0,\infty),+)$ and for each $c\in \N$, $cS$ is an $IP^{*}$ set near zero in $S$. Let $k\in \N$ and let $\vec{a}=<a_{1},a_{2},\ldots,a_{k}>$ be a compressed sequence in $\Z\setminus \{0\}$ with $a_{1}>0$, $a_{k}>0$. Let $p$ be an idempotent in $O^{+}(S)$. Then
\\ (1) $a_{1}\cdot p+a_{2}\cdot p+\cdots+a_{k}\cdot p\in I_{0}(MT(\vec{a});S)$.
\\ (2) $\frac{a_{1}}{a_{k}}\cdot p+\frac{a_{2}}{a_{k}}\cdot p+\cdots+\frac{a_{k-1}}{a_{k}}\cdot p+p\in I_{0}(MT(\vec{a});S)$.
\end{lem}
\begin{proof}
Conclusion (1) is an immediate consequence of \cite[Theorem 5.7]{RefDN08}. For conclusion (2), let $q=\frac{a_{1}}{a_{k}}\cdot p+\frac{a_{2}}{a_{k}}\cdot p+\cdots+\frac{a_{k-1}}{a_{k}}\cdot p+p$ and let $Q\in q$. By (1) $a_{k}\cdot q\in I_{0}(MT(\vec{a});S)$. By assumption $a_{k}\cdot S\in p$. Also $a_{k}\cdot Q \in a_{k}\cdot q$. Thus by \cite[Theorem 5.7]{RefDN08}, we may pick $\vec{x}\in (a_{k} S)^{\omega}$ such that $MT(\vec{a})\vec{x}\in (a_kQ)^\omega$ then $\frac{\vec{x}}{a_{k}}\in S^\omega$ and $MT(\vec{a})\frac{\vec{x}}{a_{k}}\in Q^{\omega}$.
\end{proof}
\begin{lem}\label{3.13}
Let $u,v\in \N$ and let $D$ be a $u\times v$ image partition regular matrix. Let $\alpha$ and $\delta$ be positive ordinals, let $A$ be an $\alpha\times v$ matrix, all of whose rows are rows of $D$, and let $B$ be an $\alpha\times \delta$ matrix which is image partition regular near zero over $S$. If $r\in I_{0}(A;S)$ and $p\in I_{0}(B;S)$, then $r+p\in I_{0}(C;S)$ where $C=(A \ B)$.
\end{lem}
\begin{proof}
Proof of this lemma immediately follows from \cite[Lemma 3.6]{RefHS18}.
\end{proof}
\begin{thm}\label{3.14}
Let $S$ be a dense subsemigroup of $((0,\infty),+)$. Let $A$ be a subtracted image partition regular matrix near zero and let $M$ be a Milliken-Taylor matrix near zero. Then $\left(\begin{matrix} A & 0 \\ 0 & M \end{matrix}\right)$ is image partition regular near zero.
\end{thm}
\begin{proof}
Pick $u,v\in \N$, a $u\times v$ image partition regular matrix $D$, and an $\omega\times v$ matrix $A_{1}$ whose rows are all rows of $D$ and an $\omega\times \omega$ image partition regular matrix near zero $A_{2}$ such that
 $A=(A_{1} \ A_{2})$. Note that $I_{0}(A_{1};S)=I_{0}(D;S)$, so by Theorem \ref{2.16}, pick $p\in I_{0}(D;S)\cap I_{0}(A_{2};S)=I_{0}(A_{1};S)\cap I_{0}(A_{2};S)$. Choose a
  compressed sequence $<a_{i}>_{i=1}^{k}$ in $\Z\setminus \{0\}$ with $a_{1}>0$, $a_{k}>0$ such that $M=MT(\vec{a})$. We shall show that $I_{0}(A;S)\cap I_{0}(M;S)\neq \emptyset $. Now $p\in I_{0}(A_{1};S)$ and so $I_{0}(A_{1};S)+p$ is a left ideal of $I_{0}(A_{1};S)$.
 Pick an idempotent $q\in I_{0}(A_{1};S)+p$. Pick $q'\in I_{0}(A_{1};S)$ such that $q=q'+p$. Let $r=a_{1}\cdot q+a_{2}\cdot p+\cdots+a_{k}\cdot q$. By Lemma \ref{3.12}, $r\in I_{0}(M)$, so it suffices to show that $r\in I_{0}(A;S)$. 
 By Lemma \ref{3.11}, $r\in I_{0}(A_{1};S)$. Note that $a_{k}\cdot q=a_{k}\cdot q' +a_{k}\cdot p$. Also $a_{k}\cdot q+a_{k}\cdot q=a_{k}\cdot (q+q)=a_{k}\cdot q$ so $r=r+a_{k}\cdot q=r+a_{k}\cdot q'+a_{k}\cdot p$. 
By Lemma \ref{3.9}, $r+a_{k}\cdot q' \in I_{0}(A_{1};S)$. Also $a_{1}\cdot p \in I_{0}(A_{2};S)$. 
Then by Lemma \ref{3.13}, $r=(r+a_{k}\cdot q')+a_{k}\cdot p\in I_{0}(A;S)$.
\end{proof}
As a consequence of above theorem, we have the following corollary.
\begin{cor}\label{3.15}
Let $S$ be a dense subsemigroup of $((0,\infty),+)$. Let $A$ be a subtracted centrally image partition regular matrix near zero and $M$ be a Milliken-Taylor matrix near zero. Then $\left(\begin{matrix} A & 0 \\ 0 & M \end{matrix}\right)$ is image partition regular near zero.
\end{cor}
\begin{defn}\label{3.16}
Let $S$ be a dense subsemigroup of $((0,\infty),+)$. Then $\mathcal{M}_{0}(S)=\{A$: there exist positive ordinals $\alpha$ and $\delta$ such that $A$ is an $\alpha\times \delta$ admissible matrix and whenever $T$ is a thick set near zero in $S$ and $T$ is finitely colored, there exists $\vec{x}\in S^{\delta}$ such that the entries of $A\vec{x}$ are monochromatic \}. 
\end{defn}
\begin{lem}\label{3.17}
Let $\alpha$ and $\delta$ be positive ordinals, and let $A$ be an admissible $\alpha\times \delta$  matrix. The following statements are equivalent:
\\ (a) $A\in \mathcal{M}_{0}(S)$
\\ (b) for every minimal left ideal $L$ of $O^{+}(S)$, $L\cap I_{0}(A;S)\neq \emptyset$
\\ (c) for every left ideal $L$ of $O^{+}(S)$, $L\cap I_{0}(A;S)\neq \emptyset$.
\end{lem}
\begin{proof}
(a)$\Rightarrow$ (b). Suppose $A\in \mathcal{M}_{0}(S)$ and let $L$ be a minimal left ideal of $O^{+}(S)$, and suppose that $L\cap I_{0}(A;S)=\emptyset$. For each $p\in L$, pick $C_{p}\in p$ such that there does not exist $\vec{x}\in \N^{\delta}$ such that $A\vec{x}\in C_{p}^{\alpha}$. Then $L\subseteq \bigcup_{p\in L}\bar{C_{p}}$. Minimal left ideals of $O^{+}(S)$ are closed  so pick $F\in \mathcal{P}_{f}(L)$ such that $L\subseteq \bigcup_{p\in F}\bar{C_{p}}$. Let $B=\bigcup_{p\in F}C_{p}$. Since $L\subseteq \bar{B}$, $B$ is thick. Since $A\in \mathcal{M}_{0}$, pick $\vec{x}\in S^{\delta}$ and $p\in F$ such that $A\vec{x}\in C_{p}^{\alpha}$. This is a contradiction.
\\ (b)$\Rightarrow$ (c). Trivial, since every left ideal of $O^{+}(S)$ contains a minimal left ideal.
\\ (c) $\Rightarrow$ (a). Let $B$ be a thick set near zero of $S$ and pick a left ideal $L$ of $O^{+}(S)$ such that $L\subseteq \bar{B}$. Pick $p\in L\cap I_{0}(A;S)$. Let $k\in \N$ and let $B=\bigcup_{i=1}^{k}C_{i}$. Pick $i\in \{1,2,\ldots,k\}$ such that $C_{i}\in p$. Since $p\in I_{0}(A;S)$, pick $\vec{x}\in S^{\delta}$ such that $A\vec{x}\in C_{i}^{\alpha}$.
\end{proof}
\begin{lem}\label{3.18}
Let $S$ be a dense subsemigroup of $((0,\infty),+)$. 
\\ (a) If $A$ is centrally image partition regular matrix near zero over $S$ then $A\in \mathcal{M}_{0}(S)$.
\\ (b) If for each $c\in \N$, $cS$ is an $IP^{*}$ set near zero in $S$ and $A$ be a Milliken-Taylor matrix near zero then $A\in \mathcal{M}_{0}(S)$.
\end{lem}
\begin{proof}
(a) Suppose $A$ is centrally image partition regular near zero over $S$. Let $L$ be a minimal left ideal of $O^{+}(S)$ and pick a minimal idempotent $p\in L$. Then $p\in L\cap I_{0}(A;S)$. Therefore by Lemma \ref{3.17}, $A\in \mathcal{M}_{0}(S)$.
\\ (b) Suppose we have $k\in \N$ and a compressed sequence $\vec{a}=<a_{1},a_{2},\ldots,a_{k}>$ in $\Z\setminus \{0\}$ with $a_{1}>0$, $a_{k}>0$ such that $A=MT(\vec{a})$. Let $L$ be a left ideal of $O^{+}(S)$ and let $q$ be an idempotent in $L$. By Lemma \ref{3.10}, there is an idempotent $p\in O^{+}(S)$ such that $a_{k}\cdot p=q$. Since $a_{1}>0$, $a_{1}\cdot p+a_{2}\cdot p+\cdots+a_{k-1}\cdot p\in
O^{+}(S)$ and thus $a_{1}\cdot p+a_{2}\cdot p+\cdots+a_{k}\cdot p\in L$. Now by Lemma \ref{3.12}, $a_{1}\cdot p+a_{2}\cdot p+\cdots+a_{k}\cdot p\in I_{0}(A;S)\cap L$. Hence $A\in \mathcal{M}_{0}(S)$.
\end{proof}
\begin{thm}\label{3.19}
Let $S$ be a dense subsemigroup of $((0,\infty),+)$. Let $B$ be an image partition regular matrix near zero over $S$ and let $M\in \mathcal{M}_{0}(S)$. Then $\left(\begin{matrix} \bar{1} & B & 0 \\ \bar{0} & 0 & M \end{matrix}\right)$ is image partition regular near zero over $S$.
\end{thm}
\begin{proof}
Let $\alpha$ and $\delta$ be positive ordinals and $B$ be an $\alpha\times \delta$ image partition regular matrix near zero over $S$. Pick $p\in I_{0}(A;S)$ and let $L=O^{+}(S)+p$. Pick $q\in L\cap I_{0}(M;S)$ and pick $r\in O^{+}(S)$ such that $q=r+p$. Then $r\in O^{+}(S)=I_{0}(I;S)$. So by Lemma \ref{3.13}, $q\in I_{0}(C;S)$ where $C=(I \ B)$ as required.
\end{proof}
\begin{thm}\label{3.20}
Let $S$ be a dense subsemigroup of $((0,\infty),+)$ and let $B, M\in \mathcal{M}_{0}(S)$. Then $\left(\begin{matrix} \bar{1} & B & 0 \\ \bar{0} & 0 & M \end{matrix}\right)\in \mathcal{M}_{0}$.
\end{thm}
\begin{proof}
Let $L$ be a minimal left ideal of $O^{+}(S)$. Pick $q\in L\cap I_{0}(M;S)$ and $p\in L\cap I_{0}(B;S)$. Then $q\in L=O^{+}(S)+p$ so pick $r\in O^{+}(S)$ such that $q=r+p$. By Lemma \ref{3.6}, $q\in I_{0}(C)$ where $C=(\bar{1} \ B)$.
\end{proof}
As a consequence of the above theorem, one can deduce the following corollary.
\begin{cor}\label{3.21}
Let $S$ be a dense subsemigroup of $((0,\infty),+)$, let $n\in \N$ and for each $i\in \{1,2,\ldots,n\}$ let $M_{i}\in \mathcal{M}_{0}(S)$.  Then \\ $
\left(\begin{matrix} \bar{1} & M_{1} & 0 & 0 \cdots   0 & 0 & 0 \\
                     0 & 0 & \bar{1} & M_{2}  \cdots  0 & 0 & 0 \\
                     \vdots & \vdots & \vdots & \vdots \ \ \ \ \   \vdots & \vdots & \vdots \\
                     0 & 0 & 0 & 0 \cdots  \bar{1} & M_{n-1} & 0 \\
                     0 & 0 & 0 & 0 \cdots  0 & 0 & M_{n}
\end{matrix}\right)
                                       \in \mathcal{M}_{0}(S)$.

\end{cor}
\begin{proof}
Apply Theorem \ref{3.20} $n$-times.
\end{proof}
\begin{thm}\label{3.22}
Let $S$ be a dense subsemigroup of $((0,\infty),+)$ such that for each $c\in \N$, $cS$ is an $IP^{*}$ set near zero in $S$. Let $A$ and $M$ be subtracted Milliken-Taylor matrix near zero over $S$ and a Milliken-Taylor matrix near zero respectively. Then $\left(\begin{matrix} A & 0 \\ 0 & M \end{matrix}\right)\in \mathcal{M}_{0}(S)$.
\end{thm}
\begin{proof}
Pick $u,v\in \N$, a $u\times v$ image partition regular matrix $D$, an $\omega\times v$ matrix $A_{1}$ whose rows are all rows of $D$ and an Milliken-Taylor matrix near zero $A_{2}$ such that $A=(A_{1} \ A_{2})$. 
Note that $I_{0}(A_{1};S)=I_{0}(D;S)$. Let $k,m\in \N$ and let $\vec{a}=<a_{1},a_{2},\ldots,a_{k}>$  and $\vec{b}=<b_{1},b_{2},\ldots,b_{m}>$ be compressed sequences in $\Z\setminus \{0\}$ with $a_{1}>0$, $b_{1}>0$, $a_{k}>0$, $b_{m}>0$ such that $A_{2}=MT(\vec{a})$ and $M=MT(\vec{b})$. Let $L$ be a minimal left ideal of $O^{+}(S)$ and let $p$ be an idempotent in $L$. Since $D$ is finite and image partition regular, $p\in I_{0}(D;S)=I_{0}(A_{1};S)$. 
By \cite[Exercise 4.38]{RefHS98} $L=O^{+}(S)+p$. Let $q=\frac{a_{1}}{a_{k}}\cdot p+\frac{a_{2}}{a_{k}}\cdot p+\cdots+\frac{a_{k-1}}{a_{k}}\cdot p+p$ and let $q'=\frac{b_{1}}{b_{m}}\cdot p+\frac{b_{2}}{b_{m}}\cdot p+\cdots+\frac{b_{m-1}}{b_{m}}\cdot p+p$. Then $q$ and $q'$ are in $L$. By Lemma \ref{3.11}, $q, q'\in I_{0}(D;S)=I_{0}(A_{1};S)$. By Lemma \ref{3.12}, $q\in I_{0}(MT(\vec{a});S)$ and $q'\in I_{0}(MT(\vec{b});S)$. Let $L'=L\cap I_{0}(A_{1};S)$. Then $q, q'\in L'$. By \cite[Theorem 1.65]{RefHS98} $L'$ is a minimal left ideal of $I_{0}(A_{1};S)$ so $L'=I_{0}(A_{1};S)+q.$ Pick $r\in I_{0}(A_{1};S)$ such that $q'=r+q$. By Lemma \ref{3.13}, $q'\in I_{0}(A;S)$, so $L\cap I_{0}(A;S)\cap I_{0}(M;S)\neq \emptyset$  and hence $\left(\begin{matrix} A & 0 \\ 0 & M \end{matrix}\right)\in \mathcal{M}_{0}(S)$.
\end{proof}Apply Theorem \ref{3.20} n-times.
\begin{thm}\label{3.23}
Let $S$ be a dense subsemigroup of $((0,\infty),+)$ such that for each $c\in \N$, $cS$ is an $IP^{*}$ set near zero in $S$. Let $A$ and $M$ be subtracted centrally image partition regular matrix near zero over $S$ and a Milliken-Taylor matrix near zero respectively. Then $\left(\begin{matrix} A & 0 \\ 0 & M \end{matrix}\right)\in \mathcal{M}_{0}(S)$.
\end{thm}
\begin{proof}
Let $k\in \N$ and let $\vec{a}=<a_{1},a_{2},\ldots,a_{k}>$ be a compressed sequence in $\Z\setminus \{0\}$ with $a_{1}>0$, $a_{k}>0$ such that $M=MT(\vec{a})$. Pick $u,v\in \N$, a $u\times v$ image partition regular matrix $D$, an $\omega\times v$ matrix $A_{1}$ whose rows are all rows of $D$ and an centrally image partition regular matrix near zero over $S$, $A_{2}$ such that $A=(A_{1} \ A_{2})$. 
Let $p$ be a minimal left ideal of $O^{+}(S)$ and let $p$ be an idempotent in $L$. Since $D$ is finite image partition regular, $p\in I_{0}(D;S)$. Note by \cite[Exercise 4.3.8]{RefHS98}, $L=O^{+}(S)+p$. Let $q=\frac{a_{1}}{a_{k}}\cdot p+\frac{a_{2}}{a_{k}}\cdot p+\cdots+\frac{a_{k-1}}{a_{k}}\cdot p+p$. Then $q\in L$. By Lemma \ref{3.11}, $q\in I_{0}(D;S)=I_{0}(A_{1};S)$. By Lemma \ref{3.12}, $q\in I_{0}(M;S)$. Since $p$ is an idempotent in $L$ and $L$ is minimal, $q=q+p$. Since $A_{2}$ is centrally image partition regular near zero over $S$, $p\in I_{0}(A_{2};S)$. By Lemma \ref{3.13}, $q\in I_{0}(A;S)$ so $L\cap I_{0}(A;S) \cap I_{0}(M;S)$.
\end{proof}
We now introduce two more classes of infinite image partition regular matrices near zero.
\begin{defn}\label{3.24}
Let $S$ be a dense subsemigroup of $((0,\infty),+)$.
\\ (a) $\mathcal{N}_{0}(S)=\{M: M$ is an admissible matrix and for any finite image partition regular matrix $D$, $I(D;S) \cap I(M;S)\cap K(O^{+}(S))\neq \emptyset \}$.
\\ (b) $\mathcal{R}_{0}(S)=\{M: M$ is an admissible matrix and for any finite image partition regular matrix $D$ and left ideal $L$ of $I_{0}(D;S)$,  $L\cap I_{0}(M;S)\neq \emptyset \}$.
\end{defn}
\begin{lem}\label{3.25}
Let $S$ be a dense subsemigroup of $((0,\infty),+)$ such that for each $c\in \N$, $cS$ is an $IP^{*}$ set near zero. Then $\mathcal{R}_{0}(S)\subseteq \mathcal{N}_{0}(S)$.
\end{lem}
\begin{proof}
Let $M\in \mathcal{M}_{0}(S)$ and let $D$ be a finite image partition regular matrix. Pick a minimal left ideal of $I_{0}(D;S)$. By Theorem \ref{2.13}, all minimal idempotents of $O^{+}(S)$ are in $I_{0}(D;S)$ so by \cite[Theorem 1.65]{RefHS98} pick a minimal left ideal $L'$ of $O^{+}(S)$ such that $L'\cap O^{+}(S)=L$ and in particular $L\subseteq K(O^{+}(S))$. Thus $\emptyset \neq L\cap I_{0}(M;S)=L'\cap I_{0}(D;S) \cap I_{0}(M;S) \subseteq I(D;S)\cap I(M;S)\cap K(O^{+}(S))$.
\end{proof}
\begin{lem}\label{3.26}
Let $S$ be a dense subsemigroup of $((0,\infty),+)$.
\\ (a) All centrally image partition regular matrices near zero over $S$ are in $\mathcal{R}_{0}(S)$ (and thus in $\mathcal{N}_{0}(S))$.
\\ (b) If for each $c\in \N$, $cS$ is an $IP^{*}$ set near zero in $S$, then all Milliken-Taylor matrices near zero are in $\mathcal{R}_{0}(S)$ (and thus in $\mathcal{N}_{0}(S))$.
\end{lem}
\begin{proof}
Let $D$ be a finite image partition regular matrix and let $L$ be a left ideal of $I_{0}(D;S)$. We may assume that $L$ is minimal in $I_{0}(D;S)$. By \cite[Theorem 1.65]{RefHS98}
 pick a minimal left ideal $L'$ of $O^{+}(S)$ such that $L'\cap I_{0}(D;S)=L$. 
 \\ Pick an idempotent $p\in L'$. If $M$ is centrally image partition regular near zero over $S$, then we have that $p\in L\cap I_{0}(M;S)$. 
\\ Assume we have $k\in \N$ and a compressed sequence $\vec{a}=<a_{1},a_{2},\ldots,a_{k}>$ in $\Z\setminus \{0\}$ with $a_{1}>0$, $a_{k}>0$ such that $M=MT(\vec{a})$. Let $q=\frac{a_{1}}{a_{k}}\cdot p+\frac{a_{2}}{a_{k}}\cdot p+\cdots+\frac{a_{k-1}}{a_{k}}\cdot p+p$. Since $L'=O^{+}(S)+p$, $q\in L'$ and by Lemma \ref{3.11}, $q\in I_{0}(D;S)$ and by Lemma \ref{3.12}, $q\in I_{0}(M;S)$. Therefore $q\in L\cap I_{0}(M;S)$.
\end{proof}
\begin{thm}\label{3.27}
Let $S$ be a dense subsemigroup of $((0,\infty),+)$. Let $u,v\in \N$ and let $D$ be a $u\times v$ image partition regular matrix. Let $\alpha$ and $\delta$ be positive ordinals and let $A$ be an $\alpha \times v$ matrix, whose rows are the rows of $D$. 
Let $B$ be an $\alpha \times \delta$ member of $\mathcal{R}_{0}(S)$ and let $M\in \mathcal{N}_{0}(S)$ then $\left(\begin{matrix} A & B & 0 \\ 0 & 0 & M \end{matrix}\right)\in \mathcal{N}_{0}(S)$.
\end{thm}
\begin{proof}
Let $q\in I_{0}(A;S)\cap I_{0}(M;S)\cap K(O^{+}(S))$. Let $1=I_{0}(A;S)+q$. Since $B\in \mathcal{R}_{0}(S)$, pick $r\in L\cap I_{0}(B;S)$. Since $q\in K(O^{+}(S))\cap I(A;S)=K(I_{0}(A;S))$ (by \cite[Theorem 1.65] {RefHS98}). $L$ is a minimal left ideal of $I_{0}(A;S)$, so $L=L+r$. 
Pick $s\in L$ such that $q=s+r$. By Lemma \ref{3.13}, $q\in I_{0}(C;S)$ where $C=(A \ B)$ . So $q\in I_{0}(N;S)$ where $N=\left(\begin{matrix} A & B & 0 \\ 0 & 0 & M \end{matrix}\right)$. Thus $I(A;S)\cap I(M;S)\cap K(O^{+}(S))\subseteq I_{0}(N;S)$. To see that $N\in \mathcal{N}_{0}(S)$, let $E$ be a finite image partition regular matrix. 
We need to show that $I(E)\cap I(N)\cap K(O^{+}(S))\neq \emptyset$. Let $G=\left(\begin{matrix} A & 0 \\ 0 & E \end{matrix}\right)$, then the rows of $G$ are the rows of a finite image partition regular matrix. Since $M\in \mathcal{N}_{0}(S) $, $I(G;S)\cap I(M;S)\cap K(O^{+}(S))\neq \emptyset $. Also $I(G;S)\cap I(M;S)\cap K(O^{+}(S))\subseteq I(A;S)\cap I(M;S)\cap K(O^{+}(S)) \subseteq I_{0}(N;S)$. Therefore $\emptyset\neq I(F;S)\cap I(M;S)\cap K(O^{+}(S))\subseteq I(E)\cap I(N)\cap K(O^{+}(S))$ and we are done.
\end{proof}
As an consequence of the above theorem, we have following corollary:
\begin{cor}\label{3.28}
Let $S$ be a dense subsemigroup of $((0,\infty),+)$. Let $m\in \N$ and for $i\in \{1,2,\ldots,m\}$ let $M_{i}\in \mathcal{R}_{0}(S)$, let $A_{i}$ be a matrix with rows indexed by the same positive ordinals as the rows of $M_{i}$, and assume that the rows of $A_{i}$ are the rows of a finite image partition regular matrix. Let $M_{m+1}\in \mathcal{N}_{0}(S)$. 
Then $
\left(\begin{matrix} A_{1} & M_{1} & 0 & 0 \cdots   0 & 0 & 0 \\
                     0 & 0 & A_{2} & M_{2}  \cdots  0 & 0 & 0 \\
                     \vdots & \vdots & \vdots & \vdots \ \ \ \ \   \vdots & \vdots & \vdots \\
                     0 & 0 & 0 & 0 \cdots  A_{m} & M_{m} & 0 \\
                     0 & 0 & 0 & 0 \cdots  0 & 0 & M_{m+1}
\end{matrix}\right)
                                       \in \mathcal{N}_{0}(S)$.
\end{cor}
\begin{proof}
Apply Theorem \ref{3.27} $m$-times.
\end{proof}


\begin{thebibliography}{10}

\bibitem{biswaspaul} T. Biswas, D. De and R. Pal, \textit{Centrally image partition regularity near zero}, \textit{New York Journal of Mathematics}, (2015).

\bibitem{RefDN08} D. De and N. Hindman, \textit{Image partition regularity near zero}, \textit{Discrete Mathematics} (2008).

\bibitem{RefDP08} D. De and R. Pal, \textit{Universally image partition regularity}, \textit{Electronic Journal of Combinatorics}, (2008).

\bibitem{RefDHLL} W. Deuber, N. Hindman, I. Leader, H. Lefmann,
\textit{Infinite partition regular matrices}, \textit{Combinatorica}
\textbf{15} (1995), 333-355.

\bibitem{RefF} H. Furstenberg, \textit{Recurrence in Ergodic Theory
and Combinatorial Number Theory}, Princeton University Press,
Princeton, 1981.

\bibitem{RefHL99} N. Hindman, I. Leader, \textit{The Semigroup of Ultrafilter near 0}, \textit{Semigroup Forum} \textbf{59} (1999), 33-35.

\bibitem{RefHLS02} N. Hindman, I. Leader, and D. Strauss, Image
partition regular matrices-bounded solutions and prevention of
largeness, \textit{Discrete Math.} \textbf{242} (2002), 115-144.

\bibitem{RefHLS03} N. Hindman, I. Leader and D. Strauss,
\textit{Infinite image partition regular matrices - solution in central sets},
\textit{Trans. Amer. Math. Soc.} \textbf{355} (2003), 1213-1235.

\bibitem{RefHS98} N. Hindman and D. Strauss, \textit{Algebra in the
Stone-$\breve{C}$ech Compactification: Theory and Applications},
second edition, de Gruyter, Berlin, 2012.

\bibitem{RefHS00(1)} N. Hindman and D. Strauss,
\textit{Linear equation in the Stone-$\breve{C}$ech Compactification of $\N$},
\textit{Integers: The electronic journal of combinatorial number theory} \textbf{0} (2000), \# A02, 1-20.

\bibitem{RefHS18} N. Hindman and D. Strauss, \textit{Some new examples of infinite image partition regular matrices}, \textit{To be appear}.

\bibitem{RefHS00} N. Hindman and D. Strauss, \textit{Infinite partition
regular matrices, II - extending the finite results},
\textit{Topology Proc.} \textbf{25} (2000), 217-255.

\bibitem{RefPG17} S. Patra and S. Ghosh, \textit{Concerning partition regular matrices}, \textit{Integers} \textbf{17} (2017), \textit{Article A 63}.

\bibitem{RefS} I. Schur, \textit{$\ddot{U}$ber die kongruenz
$x^m + y^m = z^m$ (mod $p$)}, \textit{Jahresbericht deutschen.
Math.-Verein.} \textbf{25} (1916).

\bibitem{RefW} B. van der Waerden, \textit{Beweis einer baudetschen
vermutung}, \textit{Nieuw Arch. Wiskunde.} \textbf{19} (1927) 212-216.











\end{thebibliography}
\end{document}